\documentclass[10pt]{amsart}
\usepackage{amssymb}
\usepackage{bm}
\usepackage{graphicx}
\usepackage[centertags]{amsmath}
\usepackage{amsfonts}
\usepackage{amsthm}
\usepackage{graphicx}
\newtheorem{thm}{Theorem}
\newtheorem{cor}[thm]{Corollary}
\newtheorem{lem}[thm]{Lemma}

\newtheorem{prop}[thm]{Proposition}

\newtheorem{fact}[thm]{Fact}

\newtheorem{defn}[thm]{Definition}
\theoremstyle{definition}
\newtheorem{rem}{Remark}

\newtheorem{notation}{Notation}

\newcommand{\rr}{\mathbb{R}}
\newcommand{\nn}{\mathbb{N}}
\newcommand{\ee}{\varepsilon}
\newcommand{\ff}{\mathcal{F}}
\newcommand{\rrr}{\mathcal{R}}

\newcommand{\dens}{\mathrm{dens}}

\newcommand{\meg}{\geqslant}
\newcommand{\mik}{\leqslant}

\newcommand{\con}{\smallfrown}
\newcommand{\cv}{\mathrm{c}}

\makeindex

\begin{document}

\title{Combinatorial Structures on van der Waerden sets}

\author{Konstantinos Tyros}

\address{Department of Mathematics, University of Toronto, Toronto, Canada, M5S 2E4 }
\email{ktyros@math.toronto.edu}


\begin{abstract}
  In this paper we provide two results. The first one consists an
  infinitary version of the Furstenberg-Weiss Theorem. More precisely
  we show that every subset $A$ of a homogeneous tree $T$ such that $\frac{|A\cap T(n)|}{|T(n)|}\meg\delta$,
  where $T(n)$ denotes the $n$-th level of $T$, for all $n$ in a van der Waerden set, for some positive real
  $\delta$, contains a strong subtree having a level sets which forms a van der Waerden set.

  The second result is the following. For every sequence $(m_q)_{q}$ of positive integers and for every real $0<\delta\leqslant1$, there exists a sequence
  $(n_q)_{q}$ of positive integers such that for every $D\subseteq \bigcup_k\prod_{q=0}^{k-1}[n_q]$ satisfying
  $$\frac{\big{|}D\cap \prod_{q=0}^{k-1} [n_q]\big{|}}{\prod_{q=0}^{k-1}n_q}\geqslant\delta$$ for every $k$ in a van der Waerden set,
  there is a sequence $(J_q)_{q}$, where $J_q$ is an arithmetic progression of length $m_q$ contained in $[n_q]$ for all $q$, such that $\prod_{q=0}^{k-1}J_q\subseteq D$ for every $k$ in a van der Waerden set. Moreover, working in an abstract setting, we obtain $J_q$ to
  be any configuration of natural numbers that can be found in an arbitrary set of positive density.
\end{abstract}
\maketitle

\section{Introduction}

In 1975 E. Szemer\'edi \cite{Sz} settling in the affirmative an old conjecture
of P. Erd\H{o}s and P. Tur\'an, established a density version of van der Waerden Theorem.
More precisely, Szemer\'edi's Theorem states the following.
\begin{thm}
\label{Szemeredi}
  For every integer $k$ with $k\meg2$ and every
real $\delta$ with $0<\delta\mik1$, there exits a positive integer number $n_0$
having the following property. For every $n\meg n_0$ and every subset $A$ of
$\{0,...,n-1\}$ with at least $\delta n$ elements, there exists an arithmetic progression
of length $k$ contained in $A$. We denote the least such $n_0$ by $\mathrm{Sz}(k,\delta)$.
\end{thm}

It is an easy observation that subsets of the natural numbers of positive upper Banach density
do not necessarily contain infinite arithmetic progressions. However, using Theorem \ref{Szemeredi}
as a pigeon hole principle, one easily derives that such sets are \textit{van der Waerden},
i.e. they contain arbitrarily long arithmetic progressions.

In the present work we provide two examples of infinitary combinatorial structures that can be realized on
a van der Waerden set. The first one consists of an infinitary version of a result of Y. Furstenberg
and B. Weiss \cite{FuWe} concerning the existence of ``strong'' subtrees having an arithmetic progression
as a level set in ``big'' subsets of a homogeneous tree. Their result actually consists an extensions
of the finitary version of Szemer\'edi's Theorem (Theorem \ref{Szemeredi}). Using Furstenberg-Weiss
Theorem as a pigeon hole, we construct infinite strong subtrees having a van der Waerden set as a level set
inside ``big'' subsets of an infinite homogeneous tree.

The second result included in this work
consists of an extension of the main theorem contained in \cite{TT}, which in its turn establishes a density version of a Ramsey theoretic result \cite{DLT, To}.
In particular, it was shown that
for every positive real $\ee$ and every sequence $(m_q)_q$ of positive integers there exists a sequence $(n_q)_q$ of positive integers
with the following property. For every infinite subset $L$ of the positive integers and every sequence $(D_\ell)_{\ell\in L}$ such that
$D_\ell$ is a subset of $\prod_{q=0}^{\ell-1}\{1,...,n_q\}$ of density at least $\ee$ for all $\ell\in L$, there exist a sequence $(I_q)_q$
and infinite subset $L'$ of $L$ such that $I_q$ is a subset of $\{1,...,n_q\}$ of cardinality $m_q$ for all non-negative integers $q$ and
the set $\prod_{q=0}^{\ell-1}I_q$ is a subset of $D_\ell$ for all $\ell\in L'$.
In the present paper we provide a strengthening of this result which is optimal in several aspects. Firstly, we show that the set $L'$ can
be chosen to be a van der Waerden set provided that $L$ itself is a van der Waerden set (this is, clearly, a necessary condition). Moreover,
the sets $I_q$ can be endowed with additional structure. For instance, each $I_q$ can be chosen to be an arithmetic or a polynomial progression.
The construction of the sequence $(I_q)_q$ is effective avoiding, in particular, compactness arguments as in \cite{TT}.

To explicitly state our results we need
some pieces of notation. By $\nn$ we denote the set of the natural
numbers $\{0,1,2,...\}$ starting from $0$. By the term \emph{tree}
we mean a partial ordered set $(T,\mik)$
satisfying the following.
\begin{enumerate}
  \item[(i)] The tree $T$ is \emph{uniquely rooted}, i.e. $T$ admits a minimum with
  respect to $\mik$.
  \item[(ii)] For every $t\in T$, the set of the predecessors of $t$ in $T$
  is finite and linearly ordered.
  \item[(iii)] For every non-maximal $t\in T$, the set of the immediate successors of $t$ in $T$
  is finite.
  \item[(iv)] The tree $T$ is \emph{balanced}, i.e. all maximal chains in $T$ are
  of the same cardinality.
\end{enumerate}
For a tree $T$, we define the \emph{height} $h(T)$ of $T$ to be the cardinality of a maximal chain in $T$,
while for every node $t$ in $T$
we set $\ell_T(t)$ to be the cardinality of the set $\{s\in T:s<t\}$.
Let us point that $T$ is infinite if and only if $h(T)=\omega$, where as usual $\omega$
stands for the first infinite cardinal number.
Finally, for a tree $T$ and an integer $n$ such that $0\mik n<h(T)$, by $T(n)$ we denote the \emph{$n$-th
level} of $T$, that is the set $\{t\in T:\ell_T(t)=n\}$.
The substructures that we are interested in are the \emph{strong subtrees}, a notion that was isolated in the 1960s and it was highlighted with the work of K. Milliken in \cite{Mi1,Mi2}.
A subset $S$ of a tree $T$ is called
a strong subtree of $T$ if the following are satisfied.
\begin{enumerate}
  \item[(a)] $S$ endowed with the induced ordering from $T$ is a tree.
  \item[(b)] Every level of $S$ is a subset of some level of $T$.
  \item[(c)] Every non-maximal $s$ in $S$ has the following property.
  For every $s'$ immediate successor of $s$ in $S$ there exists an immediate successor
  $t'$ of $s$ in $T$ such that $t'\mik s'$.
\end{enumerate}
The level set of a strong subtree $S$ of a tree $T$, denoted by $L_T(S)$, is the set
$\{n\in\nn:n<h(T)\;\text{and}\;T(n)\cap S\neq\emptyset\}$.
If $b$ is a positive integer, then a
$b$-homogeneous tree $T$ is a tree such that every non-maximal node $t$ in $T$
has exactly $b$ immediate successors. A tree is called homogeneous, if it is
$b$-homogeneous for some positive integer $b$.
The result of H. Furstenberg and B. Weiss that we mentioned above is the following.
\begin{thm}
  \label{Furst_Weiss_finite}
  Let $b,k$ be positive integers and $\delta$ be a real with $0<\delta\mik1$.
  Then there exists a positive integer $n_0$ with the following property.
  For every finite $b$-homogeneous tree $T$ of height at least $n_0$ and every
  subset $A$ of $T$ satisfying
  \begin{equation}
    \label{aeq001}
    \mathbb{E}_{0\mik n<h(T)} \frac{|T(n)\cap A|}{|T(n)|}\meg\delta,
  \end{equation}
  there exists a strong subtree $S$ of $T$ contained in $A$ and having
  a level set which forms an arithmetic progression of length $k$.
  We denote the least such $n_0$ by $\mathrm{FW}(b,k,\delta)$.
\end{thm}
In \cite{PST} J. Pach, J. Solymosi and G. Tardos offered a combinatorial proof
of Theorem \ref{Furst_Weiss_finite} yielding
explicit bounds for the numbers $\mathrm{FW}(b,k,\delta)$.
In the present work we establish the following infinitary version of
Furstenberg-Weiss Theorem.
\begin{thm}
  \label{Furst_Weiss_infinite_cor}
  Let  $T$ be an infinite homogeneous tree and
  $A$ be a subset of $T$ such that
  \begin{equation}
  \label{aeq002}
    \limsup_{N\to\infty}\mathbb{E}_{0\mik n< N}\frac{|T(n)\cap A|}{|T(n)|}>0.
  \end{equation}
  Then there exists a strong subtree $S$ of $T$ contained
  in $A$ and having a level set which forms a van der Waerden set.
\end{thm}

Actually, the above result is an easy consequence of the following, which
we will prove in this paper.
\begin{thm}
  \label{Furst_Weiss_infinite}
  Let  $T$ be an infinite homogeneous tree and
  $A$ be a subset of $T$ such that there exist a van der Waerden set $L$ and
  a real $\delta$ with $0<\delta\mik1$ such that
  \begin{equation}
    \label{aeq003}
    \frac{|T(n)\cap A|}{|T(n)|}\meg\delta,
  \end{equation}
  for all $n\in L$. Then there exists a strong subtree $S$ of $T$ contained
  in $A$ and having a level set which forms a van der Waerden set.
\end{thm}
\noindent In Section \ref{section_Uniform_FurstWeiss} we obtain a uniformized version of the
Furstenberg-Weiss Theorem needed for the proof of Theorem \ref{Furst_Weiss_infinite} which is given in Section \ref{section_proof_FurstWeiss_infinite}.

To state our second result we need some additional notation.
We denote the set of the all positive integers by $\nn_+$. For every set $X$ by $X^{<\omega}$ we denote the set of all finite sequences in $X$.
The empty sequence is denoted by $\varnothing$ and is included in $X^{<\omega}$. For every positive integer $k$ by $[k]$ we denote the set
$\{1,...,k\}$. By convention, $[0]$ stands for the empty set.

Let $Y$ be a (possibly infinite) set. If $X$ is a nonempty finite subset of $Y$ and $A$ is an arbitrary subset of $Y$, then the
\textit{density of $A$ relative to $X$}, denoted by $\dens_X(A)$, is the quantity defined by
\begin{equation} \label{beq01}
\dens_X(A)=\frac{|A\cap X|}{|X|}.
\end{equation}

We will also need the following slight variant of the standard notion of a density regular family (see, e.g., \cite{Mc1,Mc2}).
\begin{defn} \label{un_dens_reg_defn}
A family $\ff$ of nonempty finite subsets of $\nn_+$ is called \emph{uniformly density regular} if for every $0<\ee\mik 1$ there exists
an integer $n_0$ such that for every interval $I$ of $\nn_+$ of length at least $n_0$ and every subset $A$ of $I$ with $\dens_I(A)\meg\ee$, the set
$A$ contains an element of $\ff$. The least such $n_0$ will be denoted by $B(\ff,\ee)$. Finally, the set of all uniformly density regular
families will be denoted by $\rrr$.
\end{defn}
There are several examples of uniformly regular families. The simplest one is the set of all subsets of the positive integers with exactly
$k$ elements where $k$ is a fixed positive integer. By the famous Szemer\'edi Theorem \cite{Sz}, the set $\mathrm{AP}_k$ of all arithmetic progressions of length $k$ is also a uniformly regular family; see \cite{G} for the best known upper bounds for the numbers $B(\mathrm{AP}_k,\ee)$.
Moreover, for every choice of polynomials $p_1,...,p_k$ taking integer values on the integers and zero on zero, the family
$\big\{\{a+p_1(n),...,a+p_k(n)\}:a\in \mathbb{Z}\text{ and }n\in\mathbb{Z}\big\}$ is uniformly density regular. This is a consequence of the
work of V. Bergelson and A. Leibman \cite{BL}. More examples of uniform density regular families can be found in \cite{FrWi,BM}.

We are now ready to state our result.
\begin{thm} \label{main_result}
Let $0<\delta\mik 1$. Then there exists a map $V_{\delta}:\rrr^{<\omega}\times\nn_+^{<\omega}\to\nn$ with the following property. For every sequence
$(n_q)_q$ of positive integers, every sequence $(\ff_q)_{q}$ of uniformly regular families, every van der Waerden subset $L$ of $\nn_+$ and every sequence $(D_k)_{k\in L}$ such that
\begin{enumerate}
\item[(a)] $n_0\meg V_\delta\big(\ff_0,\varnothing\big)$,
\item[(b)] $n_q\meg V_\delta\big((\ff_p)_{p=0}^q,(n_p)_{p=0}^{q-1}\big)$ for all positive integers $q$ and
\item[(c)] $D_\ell$ is a subset of $\prod_{q=0}^{\ell-1}[n_q]$ of density at least $\delta$ for all $\ell\in L$
\end{enumerate}
there exist a sequence $(I_q)_q$ and a van der Waerden subset $L'$ of $L$ such that
\begin{enumerate}
\item[(i)] $I_q$ is an element of $\ff_q$ contained in $[n_q]$ for all $q\in\nn$ and
\item[(ii)] $\prod_{q=0}^{\ell-1}I_q\subseteq D_\ell$ for all $\ell\in L'$.
\end{enumerate}
\end{thm}

The definition of the map $V_\delta$ is based on an auxiliary map $T$ that we will define in Section \ref{section_map_T}. The proof of Theorem
\ref{main_result} as well as the definition of the map $V_\delta$ are given in Section \ref{section_map_V_and_proof}. Some preliminary work in contained in Section
\ref{section_second_result_preliminary}.

\section{Uniformized version of the Furstenberg-Weiss Theorem}\label{section_Uniform_FurstWeiss}
The aim of this section is to obtain a uniformized version of the Furstenberg-Weiss Theorem.
The machinery that we present in this section has been developed in
\cite{DKT}. First, let us recall the
definition  of the Furstenberg-Weiss measures.
For every tree $T$, subset $A$ of $T$ and finite subset $P$ of $\nn$ with $\max P<h(T)$, we define
the Furstenberg-Weiss measure
\begin{equation}
  \label{aeq004}
  d^{\mathrm{FW}}_T(A|P)=\mathbb{E}_{p\in P}\dens_{T(p)}(A),
\end{equation}
where $\dens_{T(p)}(A)$ is as defined in \eqref{beq01}.
If $T$ is finite and $P$ is equal to $\{0,...,h(T)-1\}$, then we simply write $d^{\mathrm{FW}}_T(A)$
instead of $d^{\mathrm{FW}}_T(A|P)$. Moreover, for a subset $A$ of a tree $T$ and $T'$ strong
subtree of $T$, we set $d^{\mathrm{FW}}_{T'}(A)=d^{\mathrm{FW}}_{T'}(A\cap T')$.

\begin{thm}
  \label{Uniform_Furst_Weiss_finite}
  Let $b,k$ be positive integers and $\delta$ be a real with $0<\delta\mik1$.
  Then there exists a positive integer $n_0$ with the following property.
  For every infinite $b$-homogeneous tree $T$, every
  subset $A$ of $T$ and every arithmetic progression $P$ of length at least $n_0$ satisfying
  \begin{equation}
    \label{aeq005}
    d^{\mathrm{FW}}_T(A|P)\meg\delta,
  \end{equation}
  there exists a strong subtree $S$ of $T$ contained in $A$ and having
  a level set which forms an arithmetic progression of length $k$.
  We denote the least such $n_0$ by $\mathrm{UFW}(b,k,\delta)$.
\end{thm}

Theorem \ref{Uniform_Furst_Weiss_finite} follows by Theorem \ref{Furst_Weiss_finite}
given that we may find a strong subtree $T'$ of $T$ having $P$ as a level set and
satisfying $d^{\mathrm{FW}}_{T'}(A)\meg\delta$.
In order to prove the existence of such a $T'$ we need some notation.
First we describe a finite $b$-homogeneous tree as the set of all sequences over the set
$[b]$ of length less than $h(T)$. More precisely, we define the following.
Let $b$ a positive integer and
$n$ be a non-negative integer. We set
$[b]^n$ to be the set of all maps taking values in $[b]$ and having
domain the set $\{0,...,n-1\}$ (if $n=0$, then $[b]^0=\{\varnothing\}$
where $\varnothing$ denotes the empty sequence). We also set $[b]^{<n}=\cup_{0\mik m<n}[b]^m$.
It is easy to see that for every pair of positive integers $b,n$ the set $[b]^{<n}$
endowed with the end-extension ordering $\sqsubseteq$ is a finite $b$-homogeneous tree
of height $n$. Conversely, every finite $b$-homogeneous tree $T$ is isomorphic
to $[b]^{<h(T)}$. Recall that by $[b]^{<\omega}$ we denote the set of all finite
sequences over $[b]$. It is immediate that $[b]^{<\omega}$ endowed with the end-extension ordering
is an infinite $b$-homogeneous tree and that every infinite $b$-homogeneous tree is
isomorphic to $[b]^{<\omega}$.

We also need to recall the notion of the convolution operation
introduced in \cite{DKT}. To this end we need some additional notation. Let
$P=\{p_0<p_1<...<p_{n-1}\}$ be a non-empty finite subset
of $\nn$ and $b$ be a positive integer. By $[b]^P$ we denote the set
of all maps taking values in $[b]$ and having domain the set $P$.
The \textit{canonical isomorphism} associated to $P$
is the bijection $\mathrm{I}_P:[b]^n\to [b]^P$ defined by the rule
\begin{equation}
\label{aeq006}
\mathrm{I}_P(x)(p_i) = x(i)
\end{equation}
for every $i\in\{0,...,n-1\}$. Moreover, for every $x_1\in[b]^{P_1}$ and $x_2\in[b]^{P_2}$,
where $P_1$ and $P_2$ are disjoint finite subsets of $\nn$,
by $(x_1,x_2)$ be denote the unique element $x\in[b]^{P_1\cup P_2}$ such that
$x(i)=x_1(i)$ for all $i\in P_1$ and $x(i)=x_2(i)$ for all $i\in P_2$.
Finally, for every pair of non-negative integers $m<n$ and every $x\in[b]^n$, by $R_m(x)$
we denote the initial segment of $x$ of length $m$.

\begin{defn} \label{convolution}
Let $b$ be a positive integer and $P=\{p_0<p_1<...<P_{|P|-1}\}$ be a nonempty finite subset of $\nn$.
For every $i\in\{0,...,|P|-1\}$ we set
\begin{equation} \label{aeq007}
P_i=\{p\in P: p<p_i\} \text{ and } \overline{P}_i=\{n\in\nn: n<p_i \text{ and } n\notin P_i\}.
\end{equation}
Also let $n_P=\max(P)-|P|+1$ and set
\begin{equation} \label{aeq008}
X_P=[b]^{n_P}.
\end{equation}
We define the \emph{convolution operation} $\cv_P:[b]^{<|P|}\times X_P\to [b]^{<\omega}$ associated to $P$ as follows. For every $i\in\{0,...,|P|-1\}$,
every $t\in [b]^i$ and every $x\in X_P$ we set
\begin{equation} \label{aeq009}
\cv_P(t,x)=\big(\mathrm{I}_{P_i}(t),\mathrm{I}_{\overline{P}_i}(R_{|\overline{P}_i|}(x))\big)\in [b]^{p_i}.
\end{equation}
\end{defn}
By the definitions, the following fact is immediate.
\begin{fact}
  \label{tree_property_convolution}
  Let $b$ be a positive integer and $P$ be a non-empty finite subset of $\nn$. Also let
  $X_P$ be as defined in \eqref{aeq008}. Then for every $x\in X_P$, we have that the set
  \begin{equation}
    \label{aeq010}
    \{\cv_P(t,x):t\in [b]^{<|P|}\}
  \end{equation}
  forms a strong subtree of $[b]^{<\omega}$ having $P$ as a level set.
\end{fact}
Actually, the set defined in \eqref{aeq010} is a Carlson-Simpson tree (see \cite{DKT} for the
relative definition), but this additional structure is not needed for the purposes
of the present work. We will need the following fact.

\begin{fact}
  \label{averaging}
  Let $b$ be a positive integer and $P$ be a non-empty finite subset of $\nn$.
  Also let $A$ be a subset of $[b]^{<\omega}$. Then
  \begin{equation}
    \label{aeq011}
    \mathbb{E}_{p\in P}\dens_{[b]^p}(A)=
    \mathbb{E}_{i\in\{0,...,|P|-1\}}\mathbb{E}_{t\in[b]^i}\mathbb{E}_{x\in X_P}\chi_A\big(\cv_P(t,x)\big),
  \end{equation}
  where $\chi_A$ denotes the characteristic function of $A$ and $X_P$ is as defined in \eqref{aeq008}.
\end{fact}
\begin{proof}
  Let $P=\{p_0,...,p_{|P|-1}\}$. We will prove that for every $i\in\{0,...,|P|-1\}$
  \begin{equation}
    \label{aeq012}
    \dens_{[b]^{p_i}}(A)=
    \mathbb{E}_{t\in[b]^i}\mathbb{E}_{x\in X_P}\chi_A\big(\cv_P(t,x)\big).
  \end{equation}
  Obviously the result follows by \eqref{aeq012}. Let us fix some $i\in\{0,...,|P|-1\}$.
  First observe that for every $s\in[b]^{p_i}$ there exists unique $t_s\in[b]^i$ such that there
  exists $x\in X_P$ such that $\cv_P(t_s,x)=s$. For every $s\in[b]^{p_i}$, we set
  $\Omega_s=\{x\in X_P:\cv_P(t_s,x)=s\}$. Also observe that there exists unique
  $w_s\in [b]^{p_i-i}$ such that $\Omega_s=\{x\in X_P:w_s\sqsubseteq x\}$ and therefore
  \begin{equation}
    \label{aeq013}
    \mathbb{E}_{x\in X_P}\chi_{\{s\}}\big(\cv_{P}(t_s,x)\big)=b^{i-p_i}.
  \end{equation}
  Thus
  \begin{equation}
    \label{aeq014}
    \mathbb{E}_{t\in[b]^i}\mathbb{E}_{x\in X_P}\chi_{\{s\}}\big(\cv_{P}(t,x)\big)=b^{-p_i}.
  \end{equation}
  Finally, let us observe that $\cv_P$ restricted on $[b]^i\times X_P$ is
  onto $[b]^{p_i}$. Hence
  \begin{equation}
    \label{aeq015}
    \begin{split}
      \mathbb{E}_{t\in[b]^i}\mathbb{E}_{x\in X_P}\chi_A\big(\cv_P(t,x)\big)
      &=\mathbb{E}_{t\in[b]^i}\mathbb{E}_{x\in X_P}\sum_{s\in A}\chi_{\{s\}}\big(\cv_P(t,x)\big) \\
      &=\sum_{s\in A\cap [b]^{p_i}}\mathbb{E}_{t\in[b]^i}\mathbb{E}_{x\in X_P}\chi_{\{s\}}\big(\cv_P(t,x)\big)\\
      &\stackrel{\eqref{aeq014}}{=}\frac{|A\cap[b]^{p_i}|}{|[b]^{p_i}|}
      =\dens_{[b]^{p_i}}(A).
    \end{split}
  \end{equation}
  The proof is complete.
\end{proof}
By Facts \ref{tree_property_convolution} and \ref{averaging} we have the following corollary.
\begin{cor}
  \label{averaging_cor}
  Let $T$ be an infinite homogeneous tree, $P$ be a non-empty finite subset of $\nn$ and
  $A$ be a subset of $T$. Then there exists a strong subtree $T'$ of $T$ having $P$ as a level
  set and satisfying
  \begin{equation}
    \label{aeq016}
    d^{\mathrm{FW}}_{T'}(A)
    =d^{\mathrm{FW}}_T(A|P).
  \end{equation}
\end{cor}
We are ready to prove Theorem \ref{Uniform_Furst_Weiss_finite}.
\begin{proof}[Proof of Theorem \ref{Uniform_Furst_Weiss_finite}]
  We show that $\mathrm{UFW}(b,k,\delta)\mik \mathrm{FW}(b,k,\delta)$. Indeed, let us
  fix an infinite $b$-homogeneous tree $T$, a subset $A$ of $T$ and an arithmetic
  progression $P$ of length at least $\mathrm{FW}(b,k,\delta)$ such that
  \begin{equation}
    \label{aeq017}
    d^{\mathrm{FW}}_T(A|P)\meg\delta.
  \end{equation}
  Applying Corollary \ref{averaging_cor} we pass to a strong subtree $T'$
  of $T$ having $P$ as a level set and satisfying
  \begin{equation}
    \label{aeq018}
    d^{\mathrm{FW}}_{T'}(A)
    =d^{\mathrm{FW}}_T(A|P)\meg\delta.
  \end{equation}
  An application of Theorem \ref{Furst_Weiss_finite} completes the proof.
\end{proof}

By Theorem \ref{Uniform_Furst_Weiss_finite} we have the following immediate consequence.
\begin{cor}
  \label{Uniform_Furst_Weiss_finite_cor}
  Let $b,k$ be positive integers and $\delta$ be a real with $0<\delta\mik1$.
  Then for every infinite $b$-homogeneous tree $T$, every
  subset $A$ of $T$ and every arithmetic progression $P$ of $\nn$ of length at
  least $\mathrm{UFW}(b,k,\delta)$ satisfying
  \begin{equation}
    \label{aeq019}
    \frac{|T(n)\cap A|}{|T(n)|}\meg\delta,
  \end{equation}
  for all $n\in P$,
  there exists a strong subtree $S$ of $T$ contained in $A$ and having
  a level set which forms an arithmetic progression of length $k$.
\end{cor}

\begin{rem}
  Actually, Theorem \ref{Furst_Weiss_finite} and Corollary \ref{Uniform_Furst_Weiss_finite_cor}
are equivalent.
In order to establish that Corollary \ref{Uniform_Furst_Weiss_finite}
implies Theorem \ref{Furst_Weiss_finite}, let us fix a pair of positive integers $b,k$
and a real $\delta$
with $0<\delta\mik1$. Also let $T$ be a $b$-homogeneous tree of height
at least $\mathrm{Sz}\big(\mathrm{UFW}(b,k,\delta/2),\delta/2\big)$ and $A$ be a subset of
$T$ such that \eqref{aeq001} is satisfied. Setting
$L=\{n\in\{0,...,h(T)-1\}:\dens_{T(n)}(A)\meg\delta/2\}$,
by \eqref{aeq001} we have that $|L|\meg(\delta/2)h(T)$ and therefore by
Theorem \ref{Szemeredi}, there exists an arithmetic progression $P$ contained in
$L$ of length $\mathrm{UFW}(b,k,\delta/2)$. An application of Corollary \ref{Uniform_Furst_Weiss_finite_cor}
provides us with a strong subtree $S$ of $T$ contained in $A$ and having a
level set which forms an arithmetic progression of length $k$.
\end{rem}

\section{Proof of Theorem \ref{Furst_Weiss_infinite}}\label{section_proof_FurstWeiss_infinite}
Let us start by introducing some additional pieces of notation.
Let $T$ be a tree and $t$ be a node of $T$. We set $\mathrm{Succ}_T(t)=\{t'\in T:t\mik t'\}$.
Observe that $\mathrm{Succ}_T(t)$ forms a strong subtree of $T$.
Also let $n$ me a non-negative integer with $n<h(T)$ and $A$ be a subset of $T$.
We define
\begin{equation}
  \label{aeq020}
  \dens_T(A|n)=\dens_{T(n)}(A)
\end{equation}
and if $\ell_T(t)\mik n$ we set
\begin{equation}
  \label{aeq021}
  \dens_T(A|t,n)=\dens_{\mathrm{Succ}_T(t)}(A|n-\ell_T(t)).
\end{equation}
Moreover we will need the following invariants. Let $b,k$ be positive integers
and $\rho$ pe a real with $0<\rho\mik1$.
We inductively define a sequence $(f_q(b,k,\rho))_{q=1}^\infty$ as follows.
\begin{equation}
\label{aeq022}
\left\{ \begin{array} {l} f_1(b,k,\rho)=\mathrm{UFW}(b,k,\rho), \\
f_{q+1}(b,k,\rho)=\mathrm{UFW}(b,f_q(b,k,\rho),\rho). \end{array}  \right.
\end{equation}

\begin{lem}
  \label{almost_done}
  Let $b,k,q$ be positive integers and $\rho$ be a real satisfying $0<\rho\mik1$.
  Also let $T_1,...,T_q$ be infinite $b$-homogeneous trees, $P$ be an arithmetic progression of
  length $f_q(b,k,\rho)$ and $B_1,...,B_q$ be subsets of $T_1,...,T_q$ respectively
  such that $\dens_{T_j}(B_j|p)\meg\rho$ for all $j\in[q]$ and $p\in P$. Then there
  exist an arithmetic progression $Q$ of length $k$ contained in $P$ and strong subtrees
  $S_1,...,S_q$ of $T_1,...,T_q$ respectively, such that $S_j$ is contained in $B_j$
  and $L_{T_j}(S_j)=Q$ for all $j\in[q]$.
\end{lem}
\begin{proof}
  We set $P_0=P$ and we inductively construct arithmetic progressions $P_1,...,P_q$ and strong
  subtrees $S_1',...,S_q'$ of $T_1,...,T_q$ respectively satisfying the following for every $j=1,...,q$:
  \begin{enumerate}
    \item[(i)] $P_j$ is contained in $P_{j-1}$ and
    \item[(ii)] $L_{T_j}(S_j)=P_j$.
  \end{enumerate}
  The inductive step consists of a straightforward application
  of Corollary \ref{Uniform_Furst_Weiss_finite_cor}. We set $Q=P_q$ and for every $j\in[q]$ we
  pick a strong subtree $S_j$ of $S_j'$ such that $L_{T_j}(S_j)=P_j$.
  It is immediate that $S_1,...,S_q$ are as desired.
\end{proof}

We will also need the following folklore fact, essentially stating that the collection of
all van der Waerden sets forms an coideal. It is an easy consequence of the finitary version of
van der Waerden's Theorem \cite{vW}.
\begin{fact}
  \label{Ramsey}
  For every finite coloring of a van der Waerden set, one of the colors forms a van der Waerden set.
\end{fact}
\begin{lem}
  \label{sup_one_tree}
  Let $\delta$ be a real with $0<\delta\mik1$. Also let $T$ be an infinite homogeneous tree,
  $A$ be a subset of $T$, $L$ be a van der Waerden set and $t\in T$ such that
  $\dens_{T}(A|t,n)\meg\delta$ for all $n\in L$ with $n\meg\ell_T(t)$.
  Then for every integer $m$ with $m\meg\ell_T(t)$, there exist $t'\in\mathrm{Succ}_T(t)\cap T(m)$ and
  a van der Waerden subset $L'$ of $L$ such that $\min L'\meg m$ and $\dens_T(A|t',n)\meg\delta$
  for all $n\in L'$.
\end{lem}
\begin{proof}
  Let $m$ be an integer with $m\meg\ell_T(t)$. Passing to a final segment of $L$,
  if necessary, we may assume that $\min L\meg m$. Then for every
  $n\in L$ we have the following.
  By the homogeneity of $T$, we have that
  \begin{equation}
    \label{aeq023}
    \mathbb{E}_{t'\in \mathrm{Succ}_T(t)\cap T(m)}\dens_T(A|t',n)=\dens_T(A|t,n)\meg\delta.
  \end{equation}
  Thus there exists $t'_n\in \mathrm{Succ}_T(t)\cap T(m)$ such that $\dens_T(A|t'_n,n)\meg\delta$.
  Since $\mathrm{Succ}_T(t)\cap T(m)$ is finite, by Fact \ref{Ramsey} there exist
  a van der Waerden subset $L'$ of $L$ and $t'\in \mathrm{Succ}_T(t)\cap T(m)$ such that
  $t_n'=t$ for all $n\in L'$. The proof is complete.
\end{proof}

We will need the following folklore fact.
\begin{fact}
  \label{Markov}
  Let $m$ be positive integer, $\delta,\ee$ be reals with $0<\delta,\ee\mik1$ and $a_1,...,a_m$ be non-negative reals
  such that $\mathbb{E}_{j\in[m]}a_j\meg\delta$. If $a_j\mik\delta+\ee^2$ for all $j\in[m]$, then setting
  $I=\{j\in [m]: a_j\meg\delta-\ee\}$, we have that $|I|/m\meg1-\ee$.
\end{fact}
Moreover, for our convenience, for every van der Waerden set $L$, we set $\mathrm{Wd}(L)$ to be the
collection of all van der Waerden sets contained in $L$.

\begin{lem}
  \label{inductive_step}
  Let $b,k,q$ be positive integers and $\eta$ be a real with $0<\eta\mik1$.
  Also let $T_1,...,T_q$ be infinite $b$-homogeneous trees, $L$ be a van der Waerden set
  and $A_1,...,A_q$ be subsets of $T_1,...,T_q$ respectively such that
  \begin{equation}
  \label{aeq024}
    \dens_{T_j}(A_j|n)\meg \eta
  \end{equation}
  for all $j\in[q]$ and $n\in L$.
  Then there exist an arithmetic progression $Q$ contained in $L$ of length $k$,
  a van der Waerden subset $L'$ of $L$ with $\min L'\meg\max Q$ and strong subtrees
  $S_1,...,S_q$ of $T_1,...,T_q$ respectively such that for every $j\in[q]$
  the following are satisfied:
  \begin{enumerate}
    \item[(i)]  $S_j$ is contained in  $A_j$,
    \item[(ii)] $L_{T_j}(S_j)=Q$ and
    \item[(iii)] $\dens_{T_j}(A_j|t,n)\meg\eta/2$ for all $t\in S_j$ and $n\in L'$.
  \end{enumerate}
\end{lem}
\begin{proof}
  For every $j\in [q]$, $t\in T_j$ and $N\in\mathrm{Wd}(L)$ with $\min N\meg\ell_{t_j}(T)$, we set
  \begin{equation}
    \label{aeq025}
    \begin{split}
      &d_{j,t,N}=\inf_{n\in N}\dens_{T_j}(A_j|t,n)\;\text{and}\\
      &\xi_{j,N}=\sup\{d_{i,t',N'}:t'\in T_j\;\text{and}\;N'\in\mathrm{Wd}(N)\;\text{with}\;\min N'\meg\ell_{T_j}(t')\}.
    \end{split}
  \end{equation}
  Observe that for every $j\in [q]$, $t\in T_j$ and $N,N'\in\mathrm{Wd}(L)$ with $\min N,\min N'\meg\ell_{t_j}(T)$ and $N'\subseteq N$ we have that
  \begin{equation}
    \label{aeq026}
    d_{j,t,N}\mik d_{j,t,N'}.
  \end{equation}
  Also observe that for every $j\in [q]$, and $N,N'\in\mathrm{Wd}(L)$ with $N'\subseteq N$ we have that $ \xi_{j,N'}\mik \xi_{j,N}$ and therefore
  \begin{equation}
    \label{aeq027}
       \xi_{j,N'}-\xi_{j,N'}^2/18\mik  \xi_{j,N}-\xi_{j,N}^2/18.
  \end{equation}

  \noindent \textbf{Claim 1:} There exist $t_1\in T_1,...,t_q\in T_q$ with $\ell_{T_1}(t_1)=...=\ell_{T_q}(t_q)$ and $L_1\in\mathrm{Wd}(L)$ with
  $\ell_{T_1}(t_1)\mik \min L_1$ such that
  $d_{j,t_j,L_1}\meg \xi_{j,L_1}-\xi_{j,L_1}^2/18$ for all $j\in[q]$.
  \begin{proof}
    [Proof of Claim 1]
    We set $N'_0=L$. We inductively choose a decreasing sequence $(N'_j)_{j=1}^q$ in $\mathrm{Wd}(L)$ and $t'_1\in T_1,...,t_q'\in T_q$ satisfying for every $j\in[q]$ the following:
    \begin{enumerate}
      \item[(a)] $\ell_{T_j}(t_j')\mik\min N'_j$ and
      \item[(b)] $d_{j,t_j',N'_j}\meg \xi_{j,N'_{j-1}}-\xi_{j,N'_{j-1}}^2/18$
    \end{enumerate}
    The construction is immediate by \eqref{aeq025}. We set $m=\max \{\ell_{T_j}(t'_j):j\in[q]\}$ and $N_0=\{n\in N_q':n\meg m\}$.
    Applying Lemma \ref{sup_one_tree} we inductively choose a decreasing sequence $(N_j)_{j=1}^q$ in $\mathrm{Wd}(N_0)$ and
    $t_1\in\mathrm{Succ}_{T_1}(t'_1)\cap T_1(m),...,t_q\in\mathrm{Succ}_{T_q}(t'_q)\cap T_1(m)$ such that
    \begin{equation}
      \label{aeq028}
      d_{j,t_j,N_j}\meg \xi_{j,N'_{j-1}}-\xi_{j,N'_{j-1}}^2/18.
    \end{equation}
    Let $L_1=N_q$ and observe that for every $j\in [q]$ we have that $L_1\subseteq N_j\subseteq N'_j
    \subseteq N'_{j-1}$ and therefore
    \begin{equation}
      \label{aeq029}
      d_{j,t_j,L_1}\stackrel{\eqref{aeq026}}{\meg}d_{j,t_j,N_j} \stackrel{\eqref{aeq028}}{\meg} \xi_{j,N'_{j-1}}-\xi_{j,N'_{j-1}}^2/18
      \stackrel{\eqref{aeq027}}{\meg} \xi_{j,L_1}-\xi_{j,L_1}^2/18.
    \end{equation}
    Finally observe that $\min L_1\meg\min N_0\meg m=\ell_{T_1}(t_1)$. The proof of Claim 1 is complete.
  \end{proof}

  \noindent \textbf{Claim 2:} For every $p\in L_1$ and every $L_1'\in \mathrm{Wd}(L_1)$ with $p\mik\min L_1'$
  there exist $L_1''\in \mathrm{Wd}(L_1')$ and $B_1,...,B_q$ subsets of
  $A_1\cap\mathrm{Succ}_{T_1}(t_1)\cap T_1(p),...,A_q\cap\mathrm{Succ}_{T_q}(t_q)\cap T_q(p)$
  respectively such that for every $j\in[q]$ we have the following
  \begin{enumerate}
    \item[(1)] $\dens_{T_j}(B_j|t_j,p)\meg\eta/2$ and
    \item[(2)] $d_{j,t,L_1''}\meg\eta/2$ for all $t\in B_j$.
  \end{enumerate}
  \begin{proof}
    [Proof of Claim 2]
    Let $p\in L_1$ and  $L_1'\in \mathrm{Wd}(L_1)$ with $p\mik\min L_1'$.
    We set $N_0=L_1'$.
    We inductively construct a decreasing sequence $(N_j)_{j=1}^q$ in $\mathrm{Wd}(N_0)$ and $B_1,...,B_q$
    subsets of $A_1\cap\mathrm{Succ}_{T_1}(t_1)\cap T_1(p),...,A_q\cap\mathrm{Succ}_{T_q}(t_q)\cap T_q(p)$
    respectively satisfying for every $j\in [q]$ the following:
    \begin{enumerate}
      \item[($\alpha$)] $\dens_{T_j}(B_j|t_j,p)\meg\eta/2$ and
      \item[($\beta$)] $d_{j,t,N_j}\meg\eta/2$ for all $t\in B_j$.
    \end{enumerate}
    Claim 2 follows easily by the above construction and inequality \eqref{aeq026}, setting $L_1''=N_q$.
    Let us assume that for some $j\in [q]$ we have chosen $(N_i)_{i=0}^{j-1}$ and if $j>1$, $(B_i)_{i=1}^{j-1}$
    satisfying ($\alpha$) and ($\beta$) above. We describe the construction of $B_j$ and $N_j$.

    We set $\Gamma=\mathrm{Succ}_{T_j}(t_j)\cap T_j(p)$. For every $n\in N_{j-1}$ we have the following.
    By the homogeneity of the tree $T_j$, we have that
    \begin{equation}
      \label{aeq030}
      \mathbb{E}_{t\in \Gamma}\dens_{T_j}(A|t,n)=\dens_{T_j}(A|t_j,n)\meg d_{j,t_j,L_1}\meg \xi_{j,L_1}-\xi_{j,L_1}^2/18.
    \end{equation}
    The first inequality holds, since $n$ belongs to $L_1$ too, while the second one holds by Claim 1.
    Then either
    \begin{enumerate}
      \item[(a)] there exists $t_n\in \Gamma$ such that $\dens_{T_j}(A|t_n,n)\meg\xi_{j,L_1}+\xi_{j,L_1}^2/18$, or
      \item[(b)] for every $t\in \Gamma$ we have that $\dens_{T_j}(A|t,n)<\xi_{j,L_1}+\xi_{j,L_1}^2/18$.
    \end{enumerate}

    By Fact \ref{Ramsey} there exists $N'_j\in\mathrm{Wd}(N_{j-1})$ such that
    either (a) occurs for every $n\in N'_j$
    or (b) occurs for every $n\in N'_j$. We will show that (b)
    occurs for every $n\in N'_j$ by showing
    that the first alternative leads to a contradiction. Indeed, assume that (a) occurs for every $n\in N'_j$.
    By Fact \ref{Ramsey} there exist $\widetilde{N}_j\in\mathrm{Wd}(N'_j)$
    and $t\in\Gamma$ such that $t_n=t$ for
    all $n\in \widetilde{N}_j$. Hence
    $d_{j,t,\widetilde{N}_j}\meg\xi_{j,L_1}+\xi_{j,L_1}^2/18>\xi_{j,L_1}
    \meg\xi_{j,\widetilde{N}_j}$ which is a contradiction by the
    definition of $\xi_{j,L_1}$ in \eqref{aeq025}.
    Hence for every $t\in \Gamma$ we have that
    \begin{equation}
      \label{aeq031}
      \dens_{T_j}(A|t,n)<\xi_{j,L_1}+\xi_{j,L_1}^2/18
    \end{equation}
    for all $n\in N'_j$.

    By Fact \ref{Markov} applied for ``$\delta=\xi_{j,L_1}-\xi_{j,L_1}^2$'' and ``$\ee=\xi_{j,L_1}/3$''
    and inequalities \eqref{aeq030} and \eqref{aeq031},
    for every $n\in N'_j$, setting
    \begin{equation}
      \label{aeq032}
      D_n=\{t\in\Gamma:\dens_{T_j}(A_j|t,n)\meg\xi_{j,L_1}-\xi_{j,L_1}^2/18-\xi_{j,L_1}/3\},
    \end{equation}
    we have that $|D_n|/|\Gamma|\meg1-\xi_{j,L_1}/3$.
    Applying Fact \ref{Ramsey}, we obtain $N_j\in\mathrm{Wd}(N'_j)$ and $D\subseteq \Gamma$ such
    that $D_n=D$ for all $n\in N_j$. Clearly, $\dens_{T_j}(D|t_j,p)=|D|/|\Gamma|\meg1-\xi_{j,L_1}/3$ and since
    $\dens_{T_j}(A_j|t_j,p)\meg d_{j,t_j,L_1}\meg \xi_{j,L_1}-\xi_{j,L_1}^2/18$,
    setting $B_j=A_j\cap D$, we have that
    \begin{equation}
      \label{aeq033}
      \dens_{T_j}(B_j|t_j,p)\meg\xi_{j,L_1}-\xi_{j,L_1}^2/18-\xi_{j,L_1}/3\meg\xi_{j,L_1}/2
      \meg\delta/2.
    \end{equation}
    Moreover, by \eqref{aeq032}, for every $n\in N_j$, we have that
    \begin{equation}
      \label{aeq034}
      \dens_{T_j}(A_j|t,n)\meg\xi_{j,L_1}-\xi_{j,L_1}^2/18-\xi_{j,L_1}/3\meg\xi_{j,L_1}/2
      \meg\eta/2.
    \end{equation}
    The proof of the inductive step and of the inductive construction is complete.
    As we have already noticed Claim 2 follows easily by the above construction and
    inequality \eqref{aeq026}, setting $L_1''=N_q$.
  \end{proof}
  The lemma follows by an iterated use of Claim 2 and an application of Lemma \ref{almost_done}.
  Indeed, let $P$ be an arithmetic progression of length $f_q(b,k,\eta/2)$ contained in $L_1$.
  Set $L_2=\{n\in L_1:n\meg\max P\}$. Clearly, $L_2$ is a van der Waerden set.
  Making use of Claim 2, we inductively construct a decreasing sequence $(L^p)_{p\in P}$ in
  $\mathrm{Wd}(L_2)$ and sequences $(B_1^p)_{p\in P},...,(B_q^p)_{p\in P}$ of subsets of
  $A_1\cap\mathrm{Succ}_{T_1}(t_1),...,A_q\cap\mathrm{Succ}_{T_q}(t_q)$ respectively
  satisfying  for every $j\in[q]$ and $p\in P$ we have the following
  \begin{enumerate}
    \item[(1)] $B_j^p\subseteq T_j(p)$,
    \item[(2)] $\dens_{T_j}(B_j^p|t_j,p)\meg\eta/2$ and
    \item[(3)] $d_{j,t,L^p}\meg\eta/2$ for all $t\in B^p_j$.
  \end{enumerate}
  For every $j\in [q]$ we set $B_j=\cup_{p\in P}B_j^p$.
  By (2) above, we have that $\dens_{T_j}(B_j|p)\meg\eta/2$ for all $j\in[q]$ and $p\in P$.
  By Lemma \ref{almost_done}, there exist an arithmetic progression $Q$ of length $k$ contained in $P$
  and strong subtrees $S_1,...,S_q$ of $T_1,...,T_q$ respectively, such that $S_j$ is contained in $B_j$
  and $L_{T_j}(S_j)=Q$ for all $j\in[q]$. Setting $L'=L^q$, the proof is complete.
\end{proof}

Before we proceed to the proof of Theorem \ref{Furst_Weiss_infinite} let us introduce some
additional notation. Let $T$ be a tree and $m$ be a positive integer with $m\mik h(T)$.
We set
\begin{equation}
  \label{aeq035}
  T|m=\bigcup_{n=0}^{m-1}T(n).
\end{equation}
Clearly $T|m$ is strong subtree of $T$.
\begin{proof}
  [Proof of Theorem \ref{Furst_Weiss_infinite}]
  Let $b$ the branching number of $T$, i.e the positive integer such that $T$ is a $b$-homogenous tree.
  Let $h_r=1+2+...+r$ for every positive integer $r$.
  We inductively construct an increasing sequence $(W_r)_{r=1}^\infty$ of finite strong subtrees of $T$,
  a sequence $(Q_r)_{r=1}^\infty$ of arithmetic progressions contained in $L$
  and a decreasing sequence $(L_r)_{r=1}^\infty$ in $\mathrm{Wd}(L)$ satisfying for every positive integer $r$
  the following.
  \begin{enumerate}
    \item[(a)] The arithmetic progression $Q_r$ is of length $r$.
    \item[(b)] $\max Q_r<\min Q_{r+1}$.
    \item[(c)] $h(W_r)=h_r+1$ and $L_T(W_r|h_r)=\cup_{j=1}^r Q_r$.
    \item[(d)] $W_r|h_r=W_{r+1}|h_r$.
    \item[(e)] $W_r$ is contained in $A$.
    \item[(f)] $\max L_T(W_r)\mik L_r$.
    \item[(e)] For every $t\in W_r$ and $n\in L_r$ we have that $\dens_{T}(A|t,n)\meg\delta/2^r$.
  \end{enumerate}
  By a straightforward application of Lemma \ref{inductive_step} for ``$q=1$'', ``$k=2$'', ``$\eta=\delta$'', ``$T_1=T$''
  and ``$A_1=A$'' we obtain $Q_1$, $W_1$ and $L_1$ as desired.
  Assume that for some positive integer $r$ the sequences $(W_j)_{j=1}^r$, $(Q_j)_{j=1}^r$
  and $(L_j)_{j=1}^r$ have been chosen satisfying (a)-(e) above. We describe the construction
  of $W_{r+1}$, $Q_{r+1}$ and $L_{r+1}$.

  Let $q=|W(h_r)|=b^{h_r+1}$ and $(t_j)_{j=1}^q$ be an enumeration of the set $W(h_r)$.
  For every $j\in [q]$ we set $T_j=\mathrm{Succ}_{T_j}(t_j)$ and $A_j=A\cap T_j$.
  Also let $L_{r+1}'=\{n-h_r:n\in L_r\}$. Applying Lemma \ref{inductive_step}
  for ``$k=r+2$'' and ``$\eta=\delta/2^r$'', we obtain an arithmetic progression $Q'_{r+1}$
  of length $r+2$ contained in $L_{r+1}'$, a van der Waerden subset $L_{r+1}''$ of $L_{r+1}'$
  with $\max Q'_{r+1}\mik \min L_{r+1}''$ and strong subtrees $S_1,...,S_q$ of $T_1,...,T_q$
  respectively, such that for every $j\in[q]$
  the following are satisfied:
  \begin{enumerate}
    \item[(i)]  $S_j$ is contained in $A_j$,
    \item[(ii)] $L_{T_j}(S_j)=Q_{r+1}'$ and
    \item[(iii)] $\dens_{T_j}(A_j|t,n)\meg\delta/2^{r+1}$ for all $t\in S_j$ and $n\in L_{r+1}''$.
  \end{enumerate}
    We set $L_{r+1}=\{n+h_t:n\in L_{r+1}''\}$, $Q_{r+1}=\{n+h_t:n\in Q_{r+1}'\setminus\{\max Q_{r+1}'\}\}$
  and $W_{r+1}=(W_r|h_r)\cup \cup_{j=1}^q S_j$. It follows readily that $W_{r+1}$, $Q_{r+1}$ and $L_{r+1}$
  are as desired and the proof of the inductive construction is complete. Finally, setting
  $S=\cup_{r=1}^\infty (W_r|h_r)$, it is immediate that $S$ is a strong subtree of $T$
  contained in $A$ having level set $L'=\cup_{r=1}^\infty Q_r$ which is a van der Waerden subset of $L$.
\end{proof}

\section{Correlation of measurable events on arithmetic progressions.}
\label{section_second_result_preliminary}
This section contains some preliminary work for the proof of Theorem \ref{main_result}.
We proceed to define some numerical invariants. For every $0<\eta\mik 1$ and every integer $k\meg2$ we set
\begin{equation} \label{beq02}
\theta_1(k,\eta)=(\eta/2)\cdot {\mathrm{Sz}(k,\eta/2) \choose 2}^{-1},
\end{equation}
while for every $0<\eta\mik1$ we set
\begin{equation} \label{new2}
\theta_1(1,\eta)=\eta.
\end{equation}
We will need the following lemma.
\begin{lem} \label{Szem_first_cor}
Let $0<\eta\mik1$ and $k$ be a positive integer. Then for every integer $n$ with $n\meg \mathrm{Sz}(k,\eta/2)$ and every family $(A_i)_{i=1}^n$
of measurable events in a probability space $(\Omega,\Sigma,\mu)$ such that $\mu(A_i)\meg\eta$ for all $i\in [n]$ there exists an arithmetic progression $P$ of length $k$ contained in $[n]$ such that
\begin{equation} \label{beq03}
\mu\Big(\bigcap_{i\in P}A_i\Big)\meg\theta_1(k,\eta).
\end{equation}
\end{lem}
\begin{proof}
For $k=1$ the result is immediate. Thus, let as assume that $k\meg2$. Fix $n\meg \mathrm{Sz}(k,\eta/2)$ and a family $(A_i)_{i=1}^n$ of measurable events in a probability space $(\Omega,\Sigma,\mu)$ satisfying $\mu(A_i)\meg\eta$ for all $i\in [n]$. We set $n_0=\mathrm{Sz}(k,\eta/2)$ and
\begin{equation} \label{beq04}
A=\{(i,x)\in[n_0]\times\Omega:x\in A_i\}.
\end{equation}
Clearly, the product probability measure $\dens_{[n_0]}\otimes\mu$ of $A$ is at least $\eta$. For every $x\in\Omega$ let
\begin{equation} \label{beq05}
A^x=\{i\in[n_0]:(i,x)\in A\}.
\end{equation}
By Fubini Theorem, setting
\begin{equation} \label{beq06}
C=\Big\{x\in\Omega:\dens_{[n_0]}(A^x)\meg\frac{\eta}{2}\Big\},
\end{equation}
we have that $C\in \Sigma$ and $\mu(C)\meg \frac{\eta}{2}$. By Theorem \ref{Szemeredi} and the choice of $n_0$, for every $x\in C$ there exists
an arithmetic progression $P_x$ of length $k$ contained in $A^x$. Observe that for every $x\in C$ we have
\begin{equation}\label{beq07}
x\in\bigcap_{i\in P_x}A_i.
\end{equation}
There are at most ${n_0\choose 2}$ many arithmetic progressions of length $k$ contained in $[n_0]$. Therefore, there exist an arithmetic
progression $P$ of length $k$ contained in $[n_0]$ and a measurable subset $C'$ of $C$ with
\begin{equation} \label{beq08}
\mu(C')\meg\mu(C)\cdot{\mathrm{Sz}(k,\eta/2)\choose2}^{-1}\stackrel{\eqref{beq02}}{\meg}\theta_1(k,\eta)
\end{equation}
and such that $P_x=P$ for all $x\in C'$. Invoking \eqref{beq07} we see that $C'\subseteq\cap_{i\in P}A_i$. Hence
$\mu(\cap_{i\in P}A_i)\meg\mu(C')\stackrel{\eqref{beq08}}{\meg}\theta_1(k,\eta)$ and the proof is completed.
\end{proof}
We will also need a variant of Lemma \ref{Szem_first_cor} which is stated in the more general context of uniformly density regular families.
To state it we need, first, to introduce some further invariants. Specifically, for every $0<\eta\mik1$ and every uniformly density regular
family $\ff$ we set
\begin{equation} \label{new}
M(\ff,\eta)=\max \Big\{ |\{F\in\ff:F\subseteq I\}|: I \text{ is an interval of length } B(\ff,\eta)\Big\}
\end{equation}
and we define
\begin{equation} \label{beq09}
\theta_2(\ff,\eta)=\frac{\eta}{4\cdot M(\ff,\eta/4)}.
\end{equation}
\begin{lem} \label{un_dens_reg_intersections_lem}
Let $0<\eta\mik1$ and $\ff$ be a uniformly density regular family. Also let $n$ be an integer with $n\meg(2/\ee)\cdot B(\ff,\ee/4)$ and $(\Omega,\Sigma,\mu)$ be a probability space. Finally let $A$ be a subset of $[n]\times\Omega$ with $(\dens_{[n]}\otimes\mu)(A)\meg\eta$.
Then there exists an element $F$ of $\ff$ such that, setting $\widetilde{A}=\{x\in\Omega:(i,x)\in A\text{ for all }i\in F\}$, we have
\begin{equation} \label{beq10}
\mu(\widetilde{A})\meg\theta_2(\ff,\eta).
\end{equation}
\end{lem}
\begin{proof}
  We set $n_0=B(\ff,\eta/4)$. First we pick a subinterval $I$ of $[n]$ of length $n_0$ such that
  \begin{equation}
    \label{beq11}
    (\dens_{I}\otimes\mu)\big(A\cap(I\times\Omega)\big)\meg\eta/2
  \end{equation}
  as follows. We set $\ell=\big\lfloor n/ n_0\big\rfloor$ and we pick $I_1,...,I_\ell$ disjoint subintervals of $[n]$ each of length $n_0$. We set $J=\cup_{j=1}^{\ell}I_j$. By the assumptions on $n$, we have that $\dens_{[n]}([n]\setminus J)<\eta/2$. Consequently, since $I_1,...,I_\ell$ are of the same length, we have that
  \begin{equation}
    \label{beq12}
    \frac{\eta}{2}\mik(\dens_J\otimes\mu)\big(A\cap(J\times\Omega)\big)=\frac{1}{\ell}\sum_{j=1}^\ell(\dens_{I_j}\otimes\mu)\big(A\cap(I_j\times\Omega)\big).
  \end{equation}
  Hence for some $j_0\in[\ell]$ we have that $(\dens_{I_{j_0}}\otimes\mu)\big(A\cap(I_{j_0}\times\Omega)\big)\meg\eta/2$. Let $I=I_{j_0}$.

  For every $i\in I$ we set
  \begin{equation}
    \label{beq13}
    A_i=\{x\in\Omega:(i,x)\in A\}
  \end{equation}
  and for every $x\in\Omega$ we set
  \begin{equation}
    \label{beq14}
    A^x=\{i\in I: (i,x)\in A\}.
  \end{equation}
  By \eqref{beq11} and Fubini Theorem, we have that the set
  \begin{equation}
    \label{beq15}
    C=\Big\{x\in\Omega: \dens_{I}(A^x)\meg\eta/4\Big\}
  \end{equation}
  is a measurable event of probability at least $\eta/4$. Since $I$ is of length $B(\ff,\eta/4)$, for every $x\in C$ we have that there exists an element $F_x$ of $\ff$ contained in $A^x$. Let us observe that for every $x\in C$, by the definition of the set $A^x$, we have that
  \begin{equation}\label{beq16}
    x\in\bigcap_{i\in F_x}A_i.
  \end{equation}
  Since $I$ contains at most $M(\ff,\eta/4)$ elements of $\ff$, we have that there exist an element $F$ of $\ff$ contained in $I$ and a measurable subset $C'$ of $C$ such that
  \begin{equation}
    \label{beq17}
    \mu(C')\meg\frac{\mu(C)}{M(\ff,\eta/4)} \stackrel{\eqref{beq09}}{\meg}\theta_2(\ff,\eta)
  \end{equation}
  and $F_x=F$ for all $x\in C'$. Invoking \eqref{beq16} we have that $C'\subset\cap_{i\in F}A_i$. Setting $\widetilde{A}$ as in the statement, we clearly have that the intersection $\cap_{i\in F}A_i$ is subset of $\widetilde{A}$. Thus $\mu(\widetilde{A})\meg\mu(\cap_{i\in F}A_i)\meg\mu(C')\stackrel{\eqref{beq16}}{\meg}\theta_2(\ff,\eta)$ as desired.
\end{proof}

Before we proceed let us introduce some additional notation.
Let $(\Omega,\Sigma,\mu)$ be a probability space and $B$ be a measurable event of positive probability. For every $A\in \Sigma$ we set
\begin{equation}
  \label{beq18}
  \mu_B(A)=\frac{\mu(A\cap B)}{\mu(B)}.
\end{equation}
We have the following elementary fact.
\begin{fact}
  \label{fact_elementary}
  Let $0<\eta,\theta\mik1$. Also let $(\Omega,\Sigma,\mu)$ be a probability space and $A,B$ be two measurable events such that $\mu(A)\meg\eta$ and $\mu(B)\meg\theta$. If $\mu_B(A)\mik\eta/2$ then $\mu(\Omega\setminus B)\meg\eta/2$ and $\mu_{\Omega\setminus B}(A)\meg\eta+\eta\theta/2$.
\end{fact}
\begin{proof}
  Assuming that $\mu_B(A)\mik\eta/2$ we have the following. First we observe that
  \begin{equation}
    \label{beq19}
      \mu(\Omega\setminus B) \meg\mu(A\setminus B) =\mu(A)-\mu(A\cap B) \meg \mu(A)-\mu_B(A)\meg\eta/2.
  \end{equation}
  Since
  \begin{equation}
    \label{beq20}
    \begin{split}
      \eta&\mik\mu(A)=\mu(\Omega\setminus B)\cdot\mu_{\Omega\setminus B}(A)+\mu(B)\cdot\mu_B(A)\\
      &\mik\big(1-\mu(B)\big)\cdot\mu_{\Omega\setminus B}(A)+(\eta/2)\cdot\mu(B),
    \end{split}
  \end{equation}
  we have that
  \begin{equation}
    \label{beq21}
    \begin{split}
      \mu_{\Omega\setminus B}(A) &\meg\eta\cdot\frac{1-\mu(B)/2}{1-\mu(B)} =\eta\cdot\Big(1+\frac{\mu(B)/2}{1-\mu(B)}\Big)\\
      &\meg \eta+\eta\cdot\mu(B)/2\meg \eta+\eta\theta/2
    \end{split}
  \end{equation}
  as desired.
\end{proof}

\begin{lem}
  \label{main_measure_theoritic_lem}
  Let $0<\eta\leq1$ and $k$ be a positive integer. Also let $L$ be a van der Waerden set and $(A_\ell)_{\ell\in L}$ be a family of measurable events in a probability space $(\Omega,\Sigma,\mu)$ such that $\mu(A_\ell)\meg\eta$ for all $\ell\in L$. Then there exist an arithmetic progression $P$ of length $k$ contained in $L$, a van der Waerden subset $L'$ of $L$ and a subset $B$ of $\cap_{i\in P}A_i$ such that
  \begin{equation}
    \label{beq22}
     \mu(B)\meg\Big(\frac{\eta}{2}\Big)^{\big{\lfloor}\frac{2}{\eta\cdot\theta_1(k,\eta)}\big{\rfloor}-1}\theta_1(k,\eta)
  \end{equation}
  and $\mu_B(A_\ell)\meg\eta/2$,  for all $\ell\in L'$. Moreover, the set $B$ belongs to the algebra generated by the family $\{A_\ell:\ell\in L\;\text{and}\;\ell\mik\max R\}$.
\end{lem}
\begin{proof}
  We set $L_0=L$ and $\Omega_0=\Omega$. We pick a positive integer $s_0$ with $s_0\mik \big{\lfloor}\frac{2}{\eta\cdot\theta_1(k,\eta)}\big{\rfloor}$ and we construct by induction $L_1,...,L_{s_0}$ and $P_1,...,P_{s_0}$ such that setting inductively $B_t=(\cap_{i\in P_{t}}A_i)\cap\Omega_{t-1}$ and $\Omega_t=\Omega_{t-1}\setminus B_t$ for all $t=1,...,s_0$, we have that the following are satisfied.
  \begin{enumerate}
     \item[(i)] For every $t=1,...,s_0$ we have that $L_t$ is a van der Waerden subset of $L_{t-1}$.
     \item[(ii)] For every $t=1,...,s_0$ we have that $P_t$ is an arithmetic progression of length $k$ contained in $L_{t-1}$.
    \item[(iii)] For every $t=0,...,s_0-1$ we have $\mu(\Omega_t)\meg(\eta/2)^t$.
    \item[(iv)] For every $t=1,...,s_0$ we have $\mu_{\Omega_{t-1}}(B_t)\meg\theta_1(k,\eta)$.
    \item[(v)] For every $t=0,...,s_0-1$ we have $\mu_{\Omega_t}(A_\ell)\meg\eta+t\cdot\eta\cdot\theta_1(k,\eta)/2$ for all $\ell\in L_t$.
      \item[(vi)] For every $t=1,...,s_0-1$ we have that $\mu_{B_t}(A_\ell)<\eta/2$ for all $\ell\in L_t$.
      \item[(vii)]  $\mu_{B_{s_0}}(A_\ell)\meg\eta/2$ for all $\ell\in L_{s_0}$.
  \end{enumerate}
  Assume that for some $s<\big{\lfloor}\frac{2}{\eta\cdot\theta_1(k,\eta)}\big{\rfloor}$ we have constructed $(L_{t})_{t=0}^s$ and if $s\meg1$, $(P_{t})_{t=1}^s$, satisfying (i)-(vi) above. Let $(\Omega_t)_{t=0}^s$ be as defined above. By the inductive assumption (i), we have that $L_s$ is a van der Waerden set and therefore we may pick an arithmetic progression $P$ of length $\mathrm{Sz}(k,\eta/2)$ contained in $L_s$. By the inductive assumption (v) and Lemma \ref{Szem_first_cor} there exists an arithmetic progression $P_{s+1}$ of length $k$ contained in $P$ such that
  \begin{equation}
    \label{beq23}
    \mu_{\Omega_{s}}(B_{s+1})\meg\theta_1(k,\eta),
  \end{equation}
  where $B_{s+1}=(\cap_{i\in P_{s+1}}A_i)\cap \Omega_s$. By Fact \ref{Ramsey} we pass to a van der Waerden subset $L_{s+1}$ of $L_s$ such that either
  \begin{enumerate}
    \item[(a)] $\mu_{B_{s+1}}(A_\ell)\meg\eta/2$, for all $\ell\in L_{s+1}$, or
    \item[(b)] $\mu_{B_{s+1}}(A_\ell)<\eta/2$, for all $\ell\in L_{s+1}$.
  \end{enumerate}
  If (a) occurs, then we set $s_0=s+1$ and the inductive construction is complete. Let us assume that (b) holds. Then invoking \eqref{beq23} and (v) of the inductive assumptions, by Fact \ref{fact_elementary}, we have that
  \begin{equation}
    \label{beq24}
    \mu_{\Omega_{s+1}}(A_\ell)\meg\eta+(s+1)\cdot\eta\cdot\theta_1(k,\eta)/2
  \end{equation}
  for all $\ell\in L_{s+1}$, where $\Omega_{s+1}=\Omega_s\setminus B_{s+1}$. Moreover, by Fact \ref{fact_elementary} and the inductive assumptions (iii) and (v) we have that $\mu(\Omega_{s+1})\meg(\eta/2)\cdot\mu(\Omega_s)\meg(\eta/2)^{s+1}$. The inductive step of the construction is complete. Finally, let as point out that if $s=\big{\lfloor}\frac{2}{\eta\cdot\theta_1(k,\eta)}\big{\rfloor}-1$, then (a) has to occur. Indeed, assuming that (b) occurs then by \eqref{beq24} we would have that the relative probability of $A_\ell$ inside $\Omega_{s+1}$ exceeds $1$.

  Hence, setting $L'=L_{s_0}$, $P=P_{s_0}$ and $B=B_{s_0}$, we have that $L'$ is a van der Waerden subset of $L$ and $P$ is an arithmetic progression of length $k$ contained in $L$. Moreover, we have that  $B=(\cap_{i\in P}A_i)\cap\Omega_{s_0-1}\subseteq\cap_{i\in P}A_i$ and
  \begin{equation}
    \label{beq25}
    \begin{split}
      \mu(B)=&\mu(B_{s_0})= \mu_{\Omega_{s_0-1}}(B_{s_0})\cdot\mu(\Omega_{s_0-1}) \stackrel{\text{(iv),(iii)}}{\meg}\theta_1(k,\eta)\cdot(\eta/2)^{s_0-1}\\
      \meg& \theta_1(k,\eta)\cdot
    (\eta/2)^{\big{\lfloor}\frac{2}{\eta\cdot\theta_1(k,\eta)}\big{\rfloor}-1}.
    \end{split}
  \end{equation}
  By (vii), we have that $\mu_B(A_\ell)\meg\eta/2$, for all $\ell\in L'$. Finally, it is immediate that the set $B$, by its definition, belongs to the algebra generated by the family $\{A_\ell:\ell\in L\;\text{and}\;\ell\mik\max R\}$ as desired.
\end{proof}

\section{The auxiliary map $T$} \label{section_map_T}
As we have already mentioned the definition of the map $V_\delta$ makes use of an auxiliary map $T$. Recall that by $\rrr$ we denote the set of all uniformly regular families (see Definition \ref{un_dens_reg_defn}). We define the map $T:\rrr^{<\omega}\times \rr_+\to\nn$, where by $\rr_+$ we denote the set of all positive reals, as follows. Let $q$ be a non-negative integer and $((\ff_p)_{p=0}^q,\ee)$ be an element of $\rrr^{<\omega}\times\rr_+$. We inductively define $(\ee_p)_{p=0}^q$ by setting
\begin{equation}
  \label{beq26}
  \ee_0=\ee\text{ and } \ee_{p+1}=\theta_2(\ff_{p},\ee_p)
\end{equation}
for all $p=0,...,q-1$. Finally we set
\begin{equation}
  \label{beq27}
  T\big((\ff_p)_{p=0}^q,\ee\big)=\Big{\lceil}\frac{2}{\ee_q}\cdot B(\ff_q,\ee_q/4)\Big{\rceil}.
\end{equation}
We then extend $T$ on $\rrr^{<\omega}\times\rr_+$ arbitrarily.
Let us observe, for later use, that if $q$ is positive then
\begin{equation}
  \label{beq28}
  T\big((\ff_p)_{p=0}^q,\ee\big)= T\big((\ff_p)_{p=1}^q,\theta_2(\ff_0,\ee)\big).
\end{equation}
Although the following notation is quite standard in the literature, we include it below for clarity.
\begin{notation}
  Let $q_0<q_1< q_2$ be non-negative integers and $(n_q)_q$ be a sequence of positive integers. Also let $x\in\prod_{p=q_0}^{q_1-1}[n_p]$ and $y\in\prod_{p=q_1}^{q_2-1}[n_p]$. By $x^\con y$ we denote the concatenation of the sequences $x,y$, i.e. the sequence $z\in\prod_{p=q_0}^{q_2-1}[n_p]$ satisfying $z(p)=x(p)$ for all $p=q_0,\ldots,q_1-1$ and $z(p)=y(p)$ for all $p=q_1,\ldots,q_2-1$. Moreover, for
  $A\subseteq\prod_{p=q_0}^{q_1-1}[n_p]$ and $B\subseteq\prod_{p=q_1}^{q_2-1}[n_p]$ we set
  \begin{equation}
    \label{beq29}
    x^\con B=\{x^\con y:y\in B\}
  \end{equation}
  and
  \begin{equation}
    \label{beq30}
    A^\con B=\bigcup_{x\in A}x^\con B.
  \end{equation}
\end{notation}
The main property of the map $T$ that we are interested in is described by the following lemma. Similar results to this one have already been considered (see \cite{E,ES,GRS}).
\begin{lem}
  \label{T_main_property_lem}
  Let $0<\ee\mik1$ and $q$ be a non-negative integer. Also set $\ff_0,...,\ff_q$ be uniformly regular families and $n_0,...,n_q$ be integers such that $n_p\meg T\big((\ff_s)_{s=0}^p,\ee\big)$ for all $p=0,...,q$. Finally, let $D$ be a subset of $\prod_{p=0}^q[n_p]$ of density at least $\ee$. Then there exists a sequence $(I_p)_{p=0}^q$ such that
  \begin{enumerate}
    \item[(i)] $I_p$ is an element of $\ff_p$ contained in $[n_p]$ for all $p=0,...,q$ and
    \item[(ii)] $\prod_{p=0}^q I_p$ is subset of $D$.
  \end{enumerate}
\end{lem}
\begin{proof}
  We proceed by induction on $q$. First let us observe that for $q=0$ we have that $T\big((\ff_0),\ee\big)\meg B(\ff_0,\ee)$ and therefore the result follows immediately by the definition the number $B(\ff_0,\ee)$.

  Assume that the statement holds for some $q$. Fix a real $\ee$ with $0<\ee\mik1$,  uniformly density regular families $\ff_0,...,\ff_{q+1}$ and integers $n_0,...,n_{q+1}$  satisfying $n_p\meg T\big((\ff_s)_{s=0}^p,\ee\big)$ for all $p=0,...,q+1$. Finally, let $D$ be a subset of $\prod_{p=0}^{q+1}[n_p]$ of density at least $\ee$. We set $\Omega=\prod_{p=1}^{q+1}[n_p]$. Observe that $\prod_{p=0}^{q+1}[n_p]=[n_0]\times\Omega$ and  that the probability measures $\dens_{\prod_{p=0}^{q+1}[n_p]}$ and $\dens_{n_0}\otimes\dens_\Omega$ are equal. Thus $(\dens_{n_0}\otimes\dens_\Omega)(D)\meg\ee$. Since $n_0\meg T\big((\ff_0),\ee\big)\stackrel{\eqref{beq27}}{=}(2/\ee)\cdot B(\ff_0,\ee/4)$, by Lemma \ref{un_dens_reg_intersections_lem}, there exists an element $I_0$ of $\ff_0$ such that setting $\widetilde{D}=\{x\in\Omega:(i,x)\in D\text{ for all }i\in I_0\}$, we have that
  \begin{equation}
    \label{beq31}
    \mu(\widetilde{D})\meg\theta_2(\ff_0,\ee).
  \end{equation}
  By the definition of $\widetilde{D}$, it is immediate that
  \begin{equation}
    \label{beq32}
    I_0^\con\widetilde{D}\subseteq D.
  \end{equation}
   Also notice that for every $p=1,...,q+1$,
  \begin{equation}
    \label{beq33}
    n_{p}\meg T\big((\ff_s)_{s=0}^{p},\ee\big)\stackrel{\eqref{beq28}}{=} T\big((\ff_s)_{s=1}^{p},\theta_2(\ff_0,\ee)\big).
  \end{equation}
  By \eqref{beq31}, \eqref{beq33} and the inductive assumption we have that there exists a sequence $(I_p)_{p=1}^{q+1}$ such that
  \begin{enumerate}
    \item[(a)] $I_p$ is an element of $\ff_p$ contained in $[n_p]$ for all $p=1,...,q+1$ and
    \item[(b)] $\prod_{p=1}^{q+1}I_p$ is subset of $\widetilde{D}$.
  \end{enumerate}
  By (b) and \eqref{beq32}, we have that $\prod_{p=0}^{q+1}I_p\subset I_0^\con\widetilde{D}\subseteq D$ and the proof is complete.
\end{proof}
\begin{defn}
  Let $0<\ee\mik1$ and $r$ be a non-negative integer. Also let $L$ be a van der Waerden subset of $\nn_+$ and $(n_q)_q$ be a sequence of positive integers. We will say that a sequence $(D_\ell)_{\ell\in L}$ is $(r,\ee,(n_q)_q)$-dense if for every $\ell\in L$ with $\ell>r$ we have that $D_\ell$ is a subset of $\prod_{p=r}^{\ell-1}[n_p]$ of density at least $\ee$.
\end{defn}

For every $0<\ee\mik 1$ and every positive integer $k$ we define
\begin{equation}
  \label{beq34}
  \theta_3(k,\ee)=\frac{1}{2}\cdot\Big(\frac{\ee}{2}\Big)^{\big\lfloor\frac{2}{\ee\cdot\theta_1(k,\ee)}\big\rfloor}\cdot\theta_1(k,\ee).
\end{equation}
Finally, we recall that for every non-negative integer $n$ and every sequence $x$ of length at least $n$ (finite or infinite), by $R_n(x)$ we denote the initial segment of $x$ of length $n$.

\begin{lem}
  \label{inductive_step}
  Let  $0<\ee\mik1$, $r$ be a non-negative integer and $k$ be a positive integer. Also let $(\ff_q)_q$ be a sequence of uniformly density regular families and $(n_q)_q$ be a sequence of positive integers such that $n_q\meg T\big((\ff_p)_{p=r}^q,\theta_3(k,\ee)\big)$, for all $q\meg r$. Finally, let $L$ be a van der Waerden subset of $\nn_+$ and $(D_\ell)_{\ell\in L}$ be $(r,\ee,(n_q)_q)$-dense. Then there exist an arithmetic progression $P$ of length $k$ inside $L$, a van der Waerden subset $L'$ of $L$, a finite sequence $(I_p)_{p=r}^{r'-1}$ and an $(r',\ee',(n_q)_q)$-dense sequence $(\widetilde{D}_\ell)_{\ell\in L'}$, where $r'=\max P$ and $\ee'=\ee\cdot2^{-(\prod_{p=r}^{r'-1}n_p+2)}$, satisfying the following.
  \begin{enumerate}
    \item[(i)] For every $p=r,...,r'-1$, the set $I_p$ is an element of $\ff_p$ contained in $[n_p]$.
    \item[(ii)] $r<\min P$.
    \item[(iii)] For every $q\in P$, the set $\prod_{p=r}^{q-1}I_p$ is a subset of $D_q$.
    \item[(iv)] For every $\ell\in L'$, the set $(\prod_{p=r}^{r'-1}I_p)^\con\widetilde{D}_\ell$ is a subset of $D_\ell$.
  \end{enumerate}
\end{lem}
\begin{proof}
  Passing to a final segment of $L$, if it is necessary, we may assume that $r<\min L$. Let $\Omega=\prod_{p=r}^\infty[n_p]$ and $\mu$ be the Lebesgue (probability) measure on $\Omega$. Also let for every $\ell\in L$, $A_\ell=\{x\in\Omega:R_{\ell-r}(x)\in D_\ell\}$. By Lemma \ref{main_measure_theoritic_lem}, applied for ``$\eta=\ee$'', there exist an arithmetic progression $P$ of length $k$ contained in $L$, a van der Waerden subset $L''$ of $L$ and a subset $\widehat{B}$ of $\cap_{i\in P}A_i$ such that
  \begin{equation}
    \label{beq35}
     \mu(\widehat{B})\meg\Big(\frac{\ee}{2}\Big)^{\big{\lfloor}\frac{2}{\ee\cdot\theta_1(k,\ee)}\big{\rfloor}-1}\theta_1(k,\ee),
  \end{equation}
  for every $\ell\in L''$ we have $\mu_{\widehat{B}}(A_\ell)\meg\ee/2$  and the set $B$ belongs to the algebra generated by the family $\{A_\ell:\ell\in L\;\text{and}\;\ell\mik\max P\}$. Thus, setting $r'=\max P$, there exists a subset $B$ of $\prod_{p=r}^{r'-1}[n_p]$ such that
  \begin{equation}
    \label{beq36}
    \dens_{\prod_{p=r}^{r'-1}[n_p]}(B)=\mu(\widehat{B})\stackrel{\eqref{beq35}}{\meg}\Big(\frac{\ee}{2}\Big)^{\big{\lfloor}\frac{2}{\ee\cdot\theta_1(k,\ee)}\big{\rfloor}-1}\theta_1(k,\ee),
  \end{equation}
  for every $\ell\in L''$, setting $B_\ell=\{x\in\prod_{p=r}^{\ell-1}[n_p]:R_{r'-r}(x)\in B\}$, we have
  \begin{equation}
    \label{beq37}
        \dens_{B_\ell}(D_\ell)\meg\ee/2
  \end{equation}
  and
  \begin{equation}
    \label{beq38}
    R_{q-r}(x)\in D_q
  \end{equation}
   for all $q\in P$ and $x\in B$.
  Passing to a final subset $L''$, if it is necessary, we may assume that $\min L''>r'$.

  For every $\ell$ in $L''$ we have the following. For each element $y$ in $\prod_{p=r'}^{\ell-1}[n_p]$ we define $\Gamma_y=\{z\in B:z^\con y\in D_\ell\}$. By \eqref{beq37} and Fubini's Theorem, we have that the set $D'_\ell=\{y\in \prod_{p=r'}^{\ell-1}[n_p]:\dens_B(\Gamma_y)\meg\ee/4\}$ is of density at least $\ee/4$ inside $\prod_{p=r'}^{\ell-1}[n_p]$. Since $\Gamma_y$ is subset of $B$ and therefore subset of $\prod_{p=r}^{r'-1}[n_p]$, we have that there exist a subset $\Gamma_\ell$ of $B$ and a subset $\widetilde{D}_\ell$ of $D'_\ell$ of density at least $(\ee/4)\cdot2^{-\prod_{p=r}^{r'-1}n_p}$ inside $\prod_{p=r}^{r'-1}[n_p]$ such that $\Gamma_y=\Gamma_\ell$ for all $y$ in $\widetilde{D}_\ell$. Let as observe that by the choice of $\Gamma_\ell$ and $D'_\ell$ we have that
  \begin{equation}
    \label{beq39}
    \Gamma_\ell^\con\widetilde{D}_\ell \subseteq D_\ell\text{ and }\dens_B(\Gamma_\ell)\meg\ee/4.
  \end{equation}

  By Fact \ref{Ramsey} there exist a subset $\Gamma$ of $B$ and a van der Waerden subset $L'$ of $L''$, such that $\Gamma_\ell=\Gamma$ for all $\ell\in L'$. Clearly, $(\widetilde{D}_{\ell})_{\ell\in L'}$ is $(r',\ee',(n_q)_q)$-dense, where $\ee'$ is defined in the statement of the lemma.
  By \eqref{beq34}, \eqref{beq36} and \eqref{beq39} we have
  \begin{equation}
    \label{beq40}
    \dens_{\prod_{p=r}^{r'-1}[n_p]}(\Gamma)=\dens_{B}(\Gamma)\cdot\dens_{\prod_{p=r}^{r'-1}[n_p]}(B)\meg\theta_3(k,\ee).
  \end{equation}
  Moreover, by \eqref{beq39}, we have that
  \begin{equation}
    \label{beq41}
    \Gamma^\con\widetilde{D}_\ell\subseteq D_\ell
  \end{equation}
  for all $\ell\in L'$.
  Since for every $q=r,...,r'-1$ we have that $n_q\meg T\big((\ff_p)_{p=r}^q,\theta_3(k,\ee)\big)$, by \eqref{beq40} and Lemma \ref{T_main_property_lem} there exists a sequence $(I_p)_{p=r}^{r'-1}$ such that
  \begin{enumerate}
    \item[(a)] $I_p$ is an element of $\ff_p$ contained in $[n_p]$ for all $p=r,...,r'-1$ and
    \item[(b)] $\prod_{p=r}^{r'-1}I_p$ is subset of $\Gamma$.
  \end{enumerate}
  Since $\prod_{p=r}^{r'-1}I_p\subseteq\Gamma\subseteq B$, by \eqref{beq38} we have that $\prod_{p=r}^{q-1}I_p$ is subset of $D_q$ for all $q\in P$.
  By (b) and \eqref{beq41} we have that $\prod_{p=r}^{r'-1}I_p^\con \widetilde{D}_\ell$ is subset of $D_\ell$ for all $\ell\in L'$ and the proof is complete.
\end{proof}

\section{Definition of the map $V_\delta$ and the proof of Theorem \ref{main_result}}\label{section_map_V_and_proof}
Let us recall that by $\rrr$ we denote the set of all uniformly density regular families. For the sequel, let us adopt for following convection. For a sequence of positive integers $(n_q)_q$ and $r$ a non-negative integer, we consider $(n_p)_{p=r}^{r-1}$ to be the empty sequence and $\prod_{p=r}^{r-1}n_p$ to be equal to zero.
Fix some real $\delta$ with $0<\delta\leq1$. We define the map $V_\delta:\rrr^{<\omega}\times\nn_+^{<\omega}\to\nn$ as follows. For every non-negative integer $q$, every finite sequence $(n_p)_{p=0}^{q-1}$ of positive integers and every finite sequence $(\ff_p)_{p=0}^q$ of uniformly density regular families we set
\begin{equation}
  \label{beq42}
  V_\delta\big((\ff_p)_{p=0}^q,(n_p)_{p=0}^{q-1}\big)=\max_{0\mik r\mik q} T\big((\ff_p)_{p=r}^q,\theta_3(r+1,\delta\cdot2^{-(\prod_{p=0}^{r-1}n_p+2r)})\big)
\end{equation}

\begin{proof}
  [Proof of Theorem \ref{main_result}]
   Let $(n_q)_q$ be a sequence of positive integers,  $(\ff_q)_{q}$ be a sequence of uniformly regular families,  $L$ be a van der Waerden subset of $\nn_+$ and $(D_\ell)_{\ell\in L}$ be $(0,\delta,(n_q)_q)$-dense such that
   \begin{equation}
     \label{beq43}
     n_q\meg V_\delta\big((\ff_p)_{p=0}^q,(n_p)_{p=0}^{q-1}\big)
   \end{equation}
   for every non-negative integer $q$. We set $L_0=L$, $r_0=0$ and $(D^0_\ell)_{\ell\in L_0}=(D_\ell)_{\ell\in L}$. We inductively construct a sequence of arithmetic progressions $(P_n)_{n=1}^\infty$ contained in $L$, a decreasing sequence of van der Waerden sets $(L_n)_n$, a sequence $\big((I_q)_{q=r_{n-1}}^{r_{n}-1}\big)_{n=1}^\infty$ where $r_n=\max P_n$ for all positive integers $n$ and a sequence $\big((D_\ell^n)_{\ell\in L_n}\big)_n$ such that for every non-negative integer $n$ we have the following:
  \begin{enumerate}
    \item[(i)] $r_n<\min P_{n+1}$.
    \item[(ii)] $(D_\ell^n)_{\ell\in L_n}$ is $(r_n,\delta\cdot2^{-(\prod_{p=0}^{r_n-1}n_p+2r_n)},(n_q)_q)$-dense.
    \item[(iii)] If $n$ is positive, then $I_p$ is an element of $\ff_p$ contained in $[n_p]$ for every $p=r_{n-1},...,r_n-1$.
    \item[(iv)] if $n$ is positive, then $P_n$ is an arithmetic progression of length $r_{n-1}+1$ contained in $L_{n-1}$ such that $\prod_{p=0}^{q-1}I_q\subseteq D_q$ for every $q\in P_n$.
    \item[(v)] $(\prod_{p=0}^{r_n-1}I_p)^\con D_\ell^n\subseteq D_\ell$ for all $\ell\in L_n$, under the convection $\prod_{p=0}^{-1}I_p=\{\varnothing\}$.
  \end{enumerate}
  Notice first that for $n=0$ the properties (i)-(v) are satisfied. Assume that for some non-negative integer $n$ we have constructed $(L_m)_{m=0}^n$, $\big((D_\ell^m)_{\ell\in L_m}\big)_{m=0}^n$ and if $n\meg1$ we have constructed $(P_m)_{m=1}^n$ and $\big((I_q)_{q=r_{m-1}}^{r_{m}-1}\big)_{m=1}^n$ satisfying (i)-(v). Then for every integer $q$ with $q\meg r_n$ we have that
  \begin{equation}
    \label{beq44}
    \begin{split}
      n_q&\stackrel{\eqref{beq43}}{\meg}V_\delta\big((\ff_p)_{p=0}^q,(n_p)_{p=0}^{q-1}\big)\\
      &\stackrel{\eqref{beq42}}{\meg}
     T\big((\ff_p)_{p=r_n}^q,\theta_3(r_n+1,\delta\cdot2^{-(\prod_{p=0}^{r_n-1}n_p+2r_n)})\big).
    \end{split}
  \end{equation}
  By \eqref{beq44} and the inductive assumption (ii), we have that the assumptions of Lemma \ref{inductive_step} for ``$\ee=\delta\cdot2^{-(\prod_{p=0}^{r_n-1}n_p+2r_n)}$ '', ``$k=r_n+1$'', ``$r=r_n$'', ``$L=L_n$'' and ``$(D_\ell)_{\ell\in L}=(D^n_\ell)_{\ell\in L_n}$'' are satisfied. Hence there exist an arithmetic progression $P_{n+1}$ of length $r_n+1$ inside $L_n$, a van der Waerden subset $L_{n+1}$ of $L_n$, a finite sequence $(I_p)_{p=r_{n}}^{r_{n+1}-1}$ and an $(r_{n+1},\ee',(n_q)_q)$-dense sequence $(D^{n+1}_\ell)_{\ell\in L_{n+1}}$, where $r_{n+1}=\max P_{n+1}$ and
  \begin{equation}
    \label{beq45}
    \ee'=\delta\cdot2^{-(\prod_{p=0}^{r_n-1}n_p+2r_n)}\cdot2^{-(\prod_{p=r_n}^{r_{n+1}-1}n_p+2)} \meg\delta\cdot2^{-(\prod_{p=0}^{r_{n+1}-1}n_p+2r_{n+1})},
  \end{equation} satisfying the following.
  \begin{enumerate}
    \item[(a)] For all $p=r_n,...,r_{n+1}-1$, the set $I_p$ is an element of $\ff_p$ contained in $[n_p]$.
    \item[(b)] $r_n<\min P_{n+1}$.
    \item[(c)] For every $q\in P_{n+1}$, the set $\prod_{p=r_n}^{q-1}I_p$ is a subset of $D_q^n$.
    \item[(d)] For every $\ell\in L_{n+1}$, the set $(\prod_{p=r_n}^{r_{n+1}-1}I_p)^\con D^{n+1}_\ell$ is a subset of $D^n_\ell$.
  \end{enumerate}
  By (c) and the inductive assumption (v) we have for every $q\in P_{n+1}$ that the set $\prod_{p=0}^{q-1}I_p$ is a subset of $D_q$. By (d) and the inductive assumption (v) we have for every $\ell\in L_{n+1}$ that the set $(\prod_{p=r_n}^{r_{n+1}-1}I_p)^\con D^{n+1}_\ell$ is a subset of $D_\ell$. The proof of the inductive step is complete.

  We set $L'=\cup_{n=1}^\infty P_n$. Moreover, observe that $r_n$ tends to infinity as $n$ tends to infinity. Thus $L'$ is a van der Waerden set.  It is straightforward that $L'$ and $(I_q)_q$ satisfy the conclusion of the Theorem.
\end{proof}

\section{Bounds for the map $V_\delta$.}
In this section, we are interested in bounds for the map $V_\delta$. For every positive integer $m$, we denote by $\ff_{[m]}$ the family of all subsets of the positive integers with $m$ elements. It is immediate that
\begin{equation}
  \label{beq46}
  B(\ff_{[m]},\ee)=\lceil m/\ee\rceil.
\end{equation}
We set
\begin{equation}
  \label{beq47}
  \rrr_c=\{\ff_{[m]}:m\text{ is a positive integer}\}.
\end{equation}

We will also need the following remark.
\begin{rem}
  \label{rem}
  Let $\rrr'$ be a subfamily of $\rrr$ and $T':\rrr'^{<\omega}\times\rr_+\to\nn$ be a map satisfying the conclusion of Lemma \ref{T_main_property_lem} for $\ff_0,...,\ff_q$ from $\rrr'$. Also let $V'_\delta:\rrr^{<\omega}\times\nn_+^{<\omega}\to\nn$ be the map defined as in \eqref{beq42} using $T'$ instead of $T$. Then one can check that $V'_\delta$ satisfies the conclusion of Theorem \ref{main_result} under the restriction that each $\ff_q$ is chosen from $\rrr'$.
\end{rem}
We define $T':\rrr_c^{<\omega}\times\rr_+\to\nn$, setting $T'\big((\ff_{[m_p]})_{p=0}^q,\ee\big)=T_\ee\big((\ff_{[m_p]})_{p=0}^q\big)$ for every choice of non-negative integer $q$, positive integers $m_0,...,m_q$ and real $\ee$ with $0<\ee\mik1$, where $T_\ee$ is defined in equation (3) from \cite{TT}. We also define $V'_\delta:\rrr^{<\omega}\times\nn_+^{<\omega}\to\nn$ as in \eqref{beq42} using $T'$ instead of $T$. Lemma 3 from \cite{TT} yields that $T'$ satisfies the conclusion of Lemma \ref{T_main_property_lem}. Hence by Remark \ref{rem} we have that $V'_\delta$ satisfies the conclusion of Theorem \ref{main_result}. For the sequel we fix a sequence $(m_q)_{q}$ of integers and a real $\delta$ satisfying the following.
\begin{enumerate}
  \item[(i)] $0<\delta\mik1$ and
  \item[(ii)] $m_q\meg2$ for every non-negative integer $q$.
\end{enumerate}
We define a maps $f_c:\nn\to\nn$ inductively as follows.
We set
\begin{equation}
  \label{beq48}
  f_c(0)=V'_\delta\big((\ff_{[m_0]}),\varnothing\big) \text{ and }f_{c}(q+1)=V'_\delta\big((\ff_{[m_p]})_{p=0}^{q+1},(f_c(p))_{p=0}^{q}\big)
\end{equation}
for all $q=0,1,...$. In particular, we are interested in the rate of growth of the map $f_c$. We will need the following inequalities.
By Lemma 14 of \cite{TT} we have that
\begin{equation}
\label{beq49}
  T'\big((\ff_{[m_p]})_{p=0}^q,\ee\big)\mik A_2\Big(5\log_2(1/\ee)\prod_{p=0}^qm_p\Big)
\end{equation}
for every non-negative integer $q$ and every $0<\ee\mik\delta/2$, where $A_2(x)=2^x$ for every real $x$. Moreover, by Theorem 18.2 of \cite{G} we have that
\begin{equation}
  \label{beq50}
  \mathrm{Sz}(k,\ee)\mik A_2^{(3)}(\log_2 (1/\ee)A_2^{(2)}(k+9)),
\end{equation}
for all positive integers $k$ and all reals $\ee$ with $0<\ee\mik1$.

\begin{prop}
  We have that $f_c(q)\mik A_2^{(1+6q)}\big(5\big((2/\delta^2\big)+1)\log_2(2/\delta)m_0+\sum_{p=1}^qm_p\big)$, for every non-negative integer $q$.
\end{prop}
\begin{proof}
  By \eqref{beq02} and \eqref{beq50}, for every positive integer $k$ and every $0<\ee\mik1$ we have
  \begin{equation}
    \label{beq51}
    \begin{split}
    \theta_1(k,\ee)&\meg  \ee\cdot [ A_2^{(3)}(\log_2 (2/\ee)A_2^{(2)}(k+9))]^{-2}\\
    &\meg A_2\big(\log_2(1/\ee)\big)^{-1}\cdot A_2^{(2)}\big(1+A_2\big(\log_2(2/\ee)\cdot A_2^{(2)}(k+9)\big)\big)^{-1}\\
    &\meg A_2^{(2)}\big(1+A_2\big(\log_2(2/\ee)\cdot A_2^{(2)}(k+9)\big)\big)^{-1}
  \end{split}
  \end{equation}
  and therefore invoking \eqref{beq34} we have that
  \begin{equation}
    \label{beq52}
    \begin{split}
      \theta_3(k,\ee) \meg &A_2\big(1+(2/\ee)\log_2(2/\ee)A_2^{(2)}\big(1+A_2\big(\log_2(2/\ee)\cdot A_2^{(2)}(k+9)\big)\big)\big)^{-1}\\
      &\cdot A_2^{(2)}\big(1+A_2\big(\log_2(2/\ee)\cdot A_2^{(2)}(k+9)\big)\big)^{-1} \\
      \meg & A_2^{(3)}\big(3+A_2\big(\log_2(2/\ee)A_2^{(2)}(k+9)\big)\big)^{-1}.
    \end{split}
  \end{equation}
  By \eqref{beq42} and \eqref{beq48}, for every positive integer $q$, we have that
  \begin{equation}
    \label{beq53}
     \begin{split}
       f_c(q)&\mik T'\big((\ff_{[m_p]})_{p=0}^q,\theta_3(q+1,\delta\cdot2^{-(\prod_{p=0}^{q-1}f_c(p)+2q)})\big)\\
       &\stackrel{\eqref{beq49},\eqref{beq52}}{\mik} A_2\big(5A_2^{(2)}\big(3+A_2\big(\log_2(2/\ee)A_2^{(2)}(k+9)\big)\big)\prod_{p=0}^qm_p\big)\\
       &\mik A_2^{(6)}\big(\log_2(2/\delta)+\prod_{p=0}^{q-1}f_c(p)+2q+m_q\big).
     \end{split}
  \end{equation}
  By \eqref{new2} and \eqref{beq34}, we have that $\theta_3(1,\ee)\meg(\ee/2)^{(2/\ee^2)+1}$, for all $0<\ee\mik1$. Thus by \eqref{beq42}, \eqref{beq48} and \eqref{beq49}, we have that
  \begin{equation}
    \label{beq54}
    f_c(0)=T'\big((\ff_{[m_0]}),\theta_3(1,\delta)\big)\mik A_2\big(5\log_2(2/\delta)(2/\delta^2+1)m_0\big).
  \end{equation}
  By inequalities \eqref{beq53}, \eqref{beq54} and using induction on $q$, the result follows.
\end{proof}


\end{document}